\documentclass[11pt,a4paper]{article}
\usepackage[T1]{fontenc}
\usepackage{lmodern}
\usepackage{amsmath,amssymb,amsthm,mathtools}
\usepackage{microtype}
\usepackage[margin=29mm]{geometry}
\usepackage{xcolor}
\usepackage[colorlinks=true,linkcolor=blue!50!black,citecolor=blue!50!black,urlcolor=blue!50!black]{hyperref}
\numberwithin{equation}{section}
\newtheorem{theorem}{Theorem}[section]
\newtheorem{lemma}[theorem]{Lemma}

\newtheorem{corollary}[theorem]{Corollary}
\newtheorem{remark}[theorem]{Remark}
\newcommand{\primed}{\substack{n,m\ge1\\(n,m)=1}}
\title{Lucas recurrences and Lambert series in coprime double sums}
\author{Nikita Kalinin}
\date{18 October 2026}
\begin{document}
\maketitle
\begin{abstract}
For the Lucas numbers $L_0=2$, $L_1=1$, and $L_n=L_{n-1}+L_{n-2}$,
we obtain
\[
\sum_{\primed}\frac{1}{L_{2n}L_{2m}L_{2(n+m)}}=\frac1{45}.
\]
This follows from the absolutely convergent identity
\[
\sum_{\primed}\frac{q^{n+m}}
{(1+q^n)(1+q^m)(1+q^{n+m})}
=\frac{q^2}{(1-q^2)^2},\qquad |q|<1.
\]
We give a self-contained proof by local cancellation on the
Calkin--Wilf tree, as a specialization of Hata's Farey-interval
telescoping principle. Specializations yield a hyperbolic-cosine
identity
\[
\sum_{\primed}
\frac{1}{\cosh(nx)\cosh(mx)\cosh((n+m)x)}
=\frac{2}{\sinh^2(2x)},\qquad x>0
\]
and weighted sums for canonical second-order Lucas
recurrences. Decomposing arbitrary positive pairs according to their
greatest common divisor gives an unrestricted sum expressed by a
Lambert series and the quasimodular Eisenstein series $E_2$.
Conversely, M\"obius inversion recovers the primitive identity from
a double Lambert-series result of Cui, Kumar, and Singh.
\end{abstract}

\section{Introduction and historical context}\label{sec:introduction}

Elementary identities for double sums can reveal unexpected relations
among special values. In his discussion of zeta values, Zagier
\cite[\S1]{Zagier1994} gives the example
\[
f(m,n)=\frac{1}{mn^3}+\frac{1}{2m^2n^2}+\frac{1}{m^3n},
\]
for which
\begin{equation}\label{eq:zagier-local}
f(m,n)-f(m+n,n)-f(m,m+n)=\frac{1}{m^2n^2}.
\end{equation}
Summation yields $\zeta(2)^2=\tfrac52\zeta(4)$, illustrating a
telescoping mechanism that also enters the study of multiple zeta
values. In this note we explore a $q$-series counterpart of this
mechanism. The same substitutions $(m,n)\mapsto(m+n,n)$ and
$(m,n)\mapsto(m,m+n)$ lead to a coprime double-sum identity whose
specializations involve hyperbolic functions and Lucas recurrences,
while removing coprimality connects it with Lambert series and
Eisenstein series.

Such elementary telescoping identities are closely related to the
theory of multiple zeta values. In particular, the double zeta values
\[
\zeta(r,s)=\sum_{m>n\ge1}\frac{1}{m^r n^s},
\qquad r\ge2,\quad s\ge1,
\]
with integer arguments satisfy two families of relations arising from
the multiplication of series (the stuffle product) and the
iterated-integral representation (the shuffle product). The resulting
double shuffle relations, together with their regularized versions,
were studied systematically by Ihara, Kaneko, and Zagier
\cite{IKZ2006}. They provide an algebraic framework for the relations
obtained from partial-fraction identities in elementary evaluations
of zeta values.

A remarkable development was the discovery by Gangl, Kaneko, and
Zagier \cite{GKZ2006} of a connection between these relations and the
theory of modular forms. They introduced a formal double zeta space,
defined by generators subject to the double shuffle relations, and
showed that its subspace generated by formal products with even
indices is canonically isomorphic to the space of modular forms of
the corresponding weight with rational Fourier coefficients. The
connection is described explicitly in terms of period polynomials,
which satisfy the Manin relations associated with the generators of
the modular group. In particular, for every even weight $k\ge4$, the
space $S_k(\mathrm{SL}_2(\mathbb Z))$ of cusp forms gives rise to
$\dim_{\mathbb C}S_k$ linearly independent relations among double
zeta values with both arguments odd. 

We hope to develop a similar theory for the multiplicative
substitutions $(x,y)\mapsto(x,xy)$ and $(x,y)\mapsto(xy,y)$
considered in the present paper. More precisely, we would like to
understand systematically the rational potentials $H(x,y)$ for
which the coboundary
\[
\Delta H(x,y)=H(x,y)-H(x,xy)-H(xy,y)
\]
admits a factorization into functions of $x$, $y$, and $xy$, and to
obtain structural classifications and dimension formulas for the
resulting families of identities. The identity proved below provides
a concrete example in this direction. Such a theory could reveal new
examples and further relations between recurrence sums and modular
or quasimodular forms. It would be particularly interesting to
discover whether an analogue of the period-polynomial theory exists
in this setting, and whether the corresponding spaces of relations
admit a similarly natural description.

{\bf Plan of the paper}. Throughout this note, pairs $(n,m)$ are
\emph{ordered} positive integer pairs, and a pair is called primitive
when $\gcd(n,m)=1$. We reserve $L_n$ for the ordinary Lucas numbers;
their even subsequence begins $L_2,L_4,L_6,L_8=3,7,18,47$.

The common source of the evaluations is the primitive $q$-identity in
Theorem~\ref{thm:q}. Its summand is the difference between a weight
at $(n,m)$ and the weights at its two children $(n,n+m)$ and
$(n+m,m)$. Summing over the resulting tree cancels all interior
weights, and an elementary estimate shows that the remaining
boundary contribution tends to zero. Section~\ref{sec:primitive}
gives this proof and its Farey-interval interpretation.

The principal antecedent is Hata's work on Farey fractions and sums
over coprime pairs \cite{Hata1995}. His Theorem~6.1 sums the
difference between an interval weight and the weights of its two
mediant children, provided the total weight at successive generations
tends to zero. We apply this principle to an explicit rational
function of $q^n$ and $q^m$. The denominator pairs follow the same
recursion as the Calkin--Wilf tree \cite{CalkinWilf2000}.
Related telescoping arguments appear in the author's proof of Hata's
formula for Euler's constant \cite{KalininHata2025} and in his work on
positive unimodular bases and Conway topographs
\cite{Kalinin2026,KalininTopographs2025}.

Section~\ref{sec:applications} translates the primitive identity
into a hyperbolic-cosine sum and weighted identities for canonical
Lucas-type recurrences. The ordinary Lucas evaluation is a direct
specialization. Section~\ref{sec:lambert} removes coprimality by
writing every positive pair as a multiple of a primitive pair. This
is the visible-lattice-point decomposition underlying work of
Campbell \cite{Campbell1994,Campbell1998}; here it produces a Lambert
series with an $E_2$ interpretation. We also explain how a double
Lambert-series identity of Cui, Kumar, and Singh
\cite{CuiKumarSingh2026}, following work of Amdeberhan, Andrews, and
Ballantine \cite{AmdeberhanAndrewsBallantine2026}, yields the primitive
identity by M\"obius inversion.

The present note establishes the telescoping evaluations and Lambert-series
connections described above. We have not identified the exact
primitive $q$-formula in the sources discussed here.

\section{A telescoping identity over primitive pairs}\label{sec:primitive}
Starting from $(1,1)$, consider the two moves
\begin{equation}\label{eq:children}
(n,m)\longmapsto(n,n+m),\qquad(n+m,m).
\end{equation}
These are the parent--child rules of the Calkin--Wilf tree
\cite{CalkinWilf2000}.
For complex $q$ with $|q|<1$ define
\begin{align}
K_q(n,m)&=\frac{q^{n+m}}{(1+q^n)(1+q^m)(1+q^{n+m})},\label{eq:kernel}\\
H_q(n,m)&=\frac{q^{n+m}}{(1-q^{2n})(1-q^{2m})}.\label{eq:potential}
\end{align}
At $q=0$ the formulas have their evident limiting meanings.

\begin{lemma}[Local cancellation]\label{lem:local}
For $n,m\ge1$ and $|q|<1$,
\begin{equation}\label{eq:local}
K_q(n,m)=H_q(n,m)-H_q(n,n+m)-H_q(n+m,m).
\end{equation}
\end{lemma}
\begin{proof}
Put $a=q^n$ and $b=q^m$.  The right-hand side of
\eqref{eq:local} becomes
\[
\frac{ab}{(1-a^2)(1-b^2)}
-\frac{a^2b}{(1-a^2)(1-a^2b^2)}
-\frac{ab^2}{(1-b^2)(1-a^2b^2)}.
\]
Over the common denominator $(1-a^2)(1-b^2)(1-a^2b^2)$,
the numerator is $ab(1-a)(1-b)(1-ab)$.
Canceling these three factors yields $ab/[(1+a)(1+b)(1+ab)]$,
which is the left-hand side.
\end{proof}

\begin{theorem}[Primitive $q$-identity]\label{thm:q}
For every $q\in\mathbb C$ with $|q|<1$,
\begin{equation}\label{eq:q}
\boxed{\displaystyle
\sum_{\primed}\frac{q^{n+m}}
{(1+q^n)(1+q^m)(1+q^{n+m})}
=\frac{q^2}{(1-q^2)^2}.}
\end{equation}
The series converges absolutely.
\end{theorem}
\begin{proof}
Every coprime positive pair occurs exactly once in the tree
\eqref{eq:children}. Indeed, except for the root $(1,1)$, a coprime pair has
unequal coordinates; its unique parent is $(n,m-n)$ if $m>n$, or
$(n-m,m)$ if $n>m$.  Repeated subtraction eventually reaches $(1,1)$.

Sum the identity \eqref{eq:local} over all vertices through level $N$.
The positive term $H_q$ at each non-root vertex cancels with its negative
occurrence at its parent.  The remaining expression is
\[
H_q(1,1)-\sum_{(n,m)\text{ at level }N+1}H_q(n,m).
\]
Put $r=|q|<1$.  We have the uniform bound
\[
|H_q(n,m)|\le\frac{r^{n+m}}{(1-r^2)^2},
\]
and the right-hand side has a finite sum over all $n,m\ge1$.  Since
successive tree levels are disjoint, the absolute mass of level $N+1$
tends to zero. This proves the formula. The bound
$|K_q(n,m)|\le r^{n+m}/(1-r)^3$ proves absolute convergence.
\end{proof}

\begin{remark}[Farey-interval interpretation]
Let $I=[a/b,c/d]\subseteq[0,1]$ have reduced endpoints and $bc-ad=1$.
Its mediant children $I^-$ and $I^+$ have denominator pairs
$(b,b+d)$ and $(b+d,d)$. Starting from $[0/1,1/1]$ therefore
reproduces the tree \eqref{eq:children} on denominator pairs.
Set
\[
\nu_q(I)=H_q(b,d).
\]
Lemma~\ref{lem:local} identifies
$\nu_q(I)-\nu_q(I^-)-\nu_q(I^+)$ with $K_q(b,d)$, and the
convergence estimate in the proof of Theorem~\ref{thm:q} makes the
total boundary weight tend to zero. Hence that theorem is precisely
the specialization of Hata's telescoping theorem
\cite[Theorem~6.1]{Hata1995} to $\nu_q$, with root weight
$\nu_q([0,1])=H_q(1,1)$. For complex $q$, apply Hata's real-valued
statement separately to real and imaginary parts.
\end{remark}

\section{Hyperbolic and Lucas formulations}\label{sec:applications}
\subsection{Hyperbolic formulation}
\begin{theorem}[Primitive hyperbolic-cosine sum]\label{thm:hyperbolic}
For every $x>0$,
\begin{equation}\label{eq:hyperbolic}
\boxed{\displaystyle
\sum_{\primed}
\frac{1}{\cosh(nx)\cosh(mx)\cosh((n+m)x)}
=\frac{2}{\sinh^2(2x)}.}
\end{equation}
\end{theorem}
\begin{proof}
Set $q=e^{-2x}$ in Theorem~\ref{thm:q}.  Since
\[
\frac1{\cosh(kx)}=\frac{2e^{-kx}}{1+e^{-2kx}},
\]
the summand in \eqref{eq:hyperbolic} is $8K_q(n,m)$.
The result follows from
\[
8\frac{q^2}{(1-q^2)^2}=\frac{2}{\sinh^2(2x)}.
\]
\end{proof}

There is also a direct hyperbolic version of Lemma~\ref{lem:local}:
\begin{align}\label{eq:hyperboliclocal}
&\frac1{\cosh a\cosh b\cosh(a+b)}\\
&\qquad=2\left[
\frac1{\sinh(2a)\sinh(2b)}
-\frac1{\sinh(2a)\sinh(2a+2b)}
-\frac1{\sinh(2b)\sinh(2a+2b)}\right].\notag
\end{align}
For $a,b>0$ this can also be checked directly using the addition formulas
for $\sinh$ and $\cosh$.  Setting $(a,b)=(nx,mx)$ shows that
\eqref{eq:hyperbolic} telescopes without passing through $q$.

\subsection{Canonical Lucas-type recurrences}
To distinguish the general recurrence from the ordinary Lucas numbers,
write $V_n=V_n(t,s)$.
Fix real parameters
\begin{equation}\label{eq:assumption}
t>0,\qquad s\ne0,\qquad \Delta=t^2+4s>0.
\end{equation}
Define
\begin{equation}\label{eq:recurrence}
V_0=2,\quad V_1=t,\quad V_n=tV_{n-1}+sV_{n-2}\quad(n\ge2).
\end{equation}
The initial conditions are essential: we are using the canonical Lucas-type
solution of the recurrence, not an arbitrary solution.  Set
\[
\alpha=\frac{t+\sqrt\Delta}{2},\qquad
\beta=\frac{t-\sqrt\Delta}{2},\qquad q=\frac\beta\alpha.
\]
Then $\alpha>0$, $|q|<1$, $\alpha\beta=-s$, and
$V_n=\alpha^n+\beta^n=\alpha^n(1+q^n)$.

\begin{theorem}[Weighted Lucas sums]\label{thm:Lucas}
Under \eqref{eq:assumption},
\begin{equation}\label{eq:Lucas}
\boxed{\displaystyle
\sum_{\primed}\frac{(-s)^{n+m}}{V_nV_mV_{n+m}}
=\frac{s^2}{t^2(t^2+4s)}.}
\end{equation}
The series converges absolutely.
\end{theorem}
\begin{proof}
By the root representation,
\[
\frac{(-s)^{n+m}}{V_nV_mV_{n+m}}
=\frac{q^{n+m}}{(1+q^n)(1+q^m)(1+q^{n+m})}.
\]
Theorem~\ref{thm:q} gives $q^2/(1-q^2)^2$, which equals
\[
\frac{(\alpha\beta)^2}{(\alpha^2-\beta^2)^2}
=\frac{s^2}{(\alpha+\beta)^2(\alpha-\beta)^2}
=\frac{s^2}{t^2\Delta}.
\]
\end{proof}

\begin{corollary}[Negative second coefficient]\label{cor:minus}
Let $t>2$ and $V_0=2$, $V_1=t$,
$V_n=tV_{n-1}-V_{n-2}$. Then
\begin{equation}\label{eq:minus}
\boxed{\displaystyle
\sum_{\primed}\frac1{V_nV_mV_{n+m}}
=\frac1{t^2(t^2-4)}.}
\end{equation}
In particular, $V_n(3,-1)=L_{2n}$, and the sum is $1/45$.
\end{corollary}
\begin{proof}
Take $s=-1$ in Theorem~\ref{thm:Lucas}.  The factor
$(-s)^{n+m}$ is then identically $1$.
\end{proof}

\begin{remark}[Equivalent parametrization]
For $t=2\cosh x$ with $x>0$, the recurrence in
Corollary~\ref{cor:minus} gives $V_n=2\cosh(nx)$.
Dividing \eqref{eq:hyperbolic} by $8$ then recovers
\eqref{eq:minus} independently.
\end{remark}

\begin{corollary}[The plus-sign recurrence and its even subsequence]
Write $V_n=V_n(t,1)$, so $V_0=2$, $V_1=t>0$, and
$V_n=tV_{n-1}+V_{n-2}$.
Then
\begin{align}
\sum_{\primed}\frac{(-1)^{n+m}}{V_nV_mV_{n+m}}
 &=\frac1{t^2(t^2+4)},\label{eq:alternating}\\
\sum_{\primed}\frac1{V_{2n}V_{2m}V_{2(n+m)}}
 &=\frac1{t^2(t^2+4)(t^2+2)^2}.\label{eq:even}
\end{align}
For $t=1$ we have $V_n=L_n$, so these give respectively $1/5$
and the $1/45$ identity in the abstract.
\end{corollary}
\begin{proof}
The first formula is Theorem~\ref{thm:Lucas} with $s=1$.
The characteristic roots $\alpha,\beta$ have product $-1$, hence the
even subsequence has characteristic roots $\alpha^2,\beta^2$ with
sum $\alpha^2+\beta^2=t^2+2$ and product $1$.
Thus $V_{2n}$ obeys the recurrence in Corollary~\ref{cor:minus}
with parameter $T=t^2+2>2$.  Its sum equals
\[
\frac1{T^2(T^2-4)}
=\frac1{t^2(t^2+4)(t^2+2)^2}.
\]
\end{proof}

For the latter specialization, the characteristic roots of $L_n$ are
$\phi=(1+\sqrt5)/2$ and $\psi=(1-\sqrt5)/2$. The even subsequence
therefore corresponds directly to $q=(\psi/\phi)^2=\phi^{-4}$ in
Theorem~\ref{thm:q}. General background on Fibonacci and Lucas
identities can be found in \cite{Vajda1989,Koshy2001}.

\section{Unrestricted sums and Lambert series}\label{sec:lambert}
\subsection{Decomposition into primitive pairs}
Write
\[
S(q)=\sum_{\primed}K_q(n,m),\qquad
F(q)=\sum_{n,m\ge1}K_q(n,m),\qquad |q|<1.
\]

\begin{theorem}[Removing the coprimality restriction]\label{thm:all}
For every complex $q$ with $|q|<1$,
\begin{equation}\label{eq:all}
\boxed{\displaystyle
F(q)=\sum_{d\ge1}\frac{q^{2d}}{(1-q^{2d})^2}
=\sum_{N\ge1}\sigma_1(N)q^{2N},\qquad
\sigma_1(N)=\sum_{d\mid N}d.}
\end{equation}
All series converge absolutely.
\end{theorem}
\begin{proof}
Every positive pair has a unique representation $(n,m)=(da,db)$
with $d\ge1$ and $(a,b)=1$. Since $K_q(da,db)=K_{q^d}(a,b)$,
absolute convergence and Theorem~\ref{thm:q} give
\begin{equation}\label{eq:gcd-decomposition}
F(q)=\sum_{d\ge1}S(q^d)
=\sum_{d\ge1}\frac{q^{2d}}{(1-q^{2d})^2}.
\end{equation}
Finally, expanding the denominator gives
\[
\sum_{d\ge1}\frac{q^{2d}}{(1-q^{2d})^2}
=\sum_{d,k\ge1}kq^{2dk}
=\sum_{N\ge1}\sigma_1(N)q^{2N}.
\]
These rearrangements are justified by absolute convergence for $|q|<1$.
\end{proof}

The same decomposition underlies the visible-lattice-point framework
of \cite{Campbell1994,Campbell1998}. For example, for $|u|,|v|<1$,
\begin{equation}\label{eq:visible-points}
\sum_{\substack{a,b\ge1\\(a,b)=1}}
\frac{u^av^b}{1-u^av^b}
=\sum_{n,m\ge1}u^nv^m
=\frac{uv}{(1-u)(1-v)}.
\end{equation}
Here expanding each geometric series enumerates the positive integer
multiples of a primitive pair.

\subsection{Eisenstein-series interpretation and consequences}
To distinguish a power-series argument from a modular parameter, define
\begin{equation}\label{eq:E2}
\mathcal E_2(z)=1-24\sum_{N\ge1}\sigma_1(N)z^N,\qquad |z|<1.
\end{equation}
The usual quasimodular Eisenstein series of weight two is
$E_2(\tau)=\mathcal E_2(e^{2\pi i\tau})$, with
$\operatorname{Im}\tau>0$; see \cite{Zagier2008}.
Theorem~\ref{thm:all} becomes
\begin{equation}\label{eq:all-E2}
F(q)=\frac{1-\mathcal E_2(q^2)}{24}.
\end{equation}
In particular, with the convention $q=e^{2\pi i\tau}$ this is
$F(q)=(1-E_2(2\tau))/24$.

\begin{corollary}[Unrestricted Lucas and hyperbolic sums]
Under \eqref{eq:assumption}, with $V_n=V_n(t,s)$ and
$q=\beta/\alpha$ as in Section~\ref{sec:applications},
\begin{equation}\label{eq:Lucasall}
\sum_{n,m\ge1}\frac{(-s)^{n+m}}{V_nV_mV_{n+m}}
=\frac{1-\mathcal E_2(q^2)}{24}.
\end{equation}
For every $x>0$,
\begin{equation}\label{eq:hyperbolicall}
\sum_{n,m\ge1}\frac1{\cosh(nx)\cosh(mx)\cosh((n+m)x)}
=2\sum_{d\ge1}\frac1{\sinh^2(2dx)}.
\end{equation}
Both sums converge absolutely.
\end{corollary}
\begin{proof}
The root representation used in Theorem~\ref{thm:Lucas} identifies the
summand in \eqref{eq:Lucasall} with $K_q(n,m)$, so
\eqref{eq:all-E2} applies. For \eqref{eq:hyperbolicall}, set
$q=e^{-2x}$ and multiply \eqref{eq:all} by $8$, as in the proof
of Theorem~\ref{thm:hyperbolic}.
\end{proof}

\subsection{The double Lambert-series connection}
Cui, Kumar, and Singh \cite[Theorem~2]{CuiKumarSingh2026}, in their
work on a conjecture of Amdeberhan, Andrews, and Ballantine
\cite{AmdeberhanAndrewsBallantine2026}, prove
\begin{equation}\label{eq:cks-lambert}
\sum_{r,k\ge1}\frac{q^{2rk}}{(1+q^k)(1-q^{2r-1})}
=\sum_{h\ge1}\frac{q^{2h}}{(1-q^{2h})^2},\qquad |q|<1.
\end{equation}
We now show how their result also implies Theorem~\ref{thm:q}.
The elementary partial-fraction identity
\[
\frac{uv}{(1+u)(1+v)(1+uv)}
=\frac{v}{1-v^2}
\left(\frac1{1+uv}-\frac1{1+u}\right)
\]
gives
\begin{align}
F(q)
&=\sum_{m\ge1}\frac{q^m}{1-q^{2m}}
\sum_{n\ge1}\left(\frac1{1+q^{n+m}}-\frac1{1+q^n}\right)
\notag\\
&=\sum_{m\ge1}\frac{q^m}{1-q^{2m}}
\sum_{k=1}^{m}\frac{q^k}{1+q^k}.
\label{eq:cks-triangular}
\end{align}
For the second equality, take finite partial sums in $n$ and use
$(1+q^n)^{-1}\to1$. The last expression is exactly the triangular
sum in \cite[equation~(21)]{CuiKumarSingh2026}, where it equals the
left-hand side of \eqref{eq:cks-lambert}. Thus their theorem gives
\eqref{eq:all} without using the primitive evaluation.

To recover that evaluation, put $R(q)=q^2/(1-q^2)^2$.
The gcd decomposition gives $F(q)=\sum_{d\ge1}S(q^d)$, whereas
\eqref{eq:cks-lambert} and \eqref{eq:cks-triangular} give
$F(q)=\sum_{d\ge1}R(q^d)$. M\"obius inversion therefore yields
\[
S(q)=\sum_{d\ge1}\mu(d)F(q^d)
=\sum_{k\ge1}R(q^k)\sum_{d\mid k}\mu(d)=R(q),
\]
where $\mu$ is the M\"obius function. These inversions and
rearrangements converge absolutely: both $S(z)$ and $R(z)$ are
$O(|z|^2)$ near zero, and
$\sum_{d,h\ge1}|q|^{2dh}<\infty$. This establishes the claimed
connection with the primitive identity.

\begin{remark}[A totient Lambert series]
The rational value of the primitive sum also has the representation
\[
R(q)=\sum_{n\ge1}\frac{\varphi(n)q^{2n}}{1-q^{2n}},
\]
where $\varphi$ is Euler's totient function. Expanding the denominators
and using $\sum_{d\mid N}\varphi(d)=N$ proves the equality.
For further Lambert-series identities, see \cite{Schmidt2026}.
\end{remark}

\end{document}